\documentclass[11pt]{amsart} 

\usepackage{amsmath}
\usepackage{amssymb}
\usepackage{graphicx}
\usepackage[utf8]{inputenc}

\usepackage[centertags]{amsmath}
\usepackage{amsthm}
\usepackage{amsfonts}
\usepackage{cleveref}
\usepackage{xcolor}
\usepackage{float}
\usepackage{graphics}
\usepackage{tikz}
\usepackage{pgfplots}

\newtheorem{theorem}{Theorem}[section]
\newtheorem{corollary}[theorem]{Corollary}

\newtheorem{proposition}[theorem]{Proposition}
\theoremstyle{definition}
\newtheorem{definition}[theorem]{Definition}
\newtheorem{example}[theorem]{Example}
\newtheorem{remark}[theorem]{Remark}
\newtheorem{assumption}[theorem]{Assumption}

\numberwithin{equation}{section}

\DeclareMathOperator{\ran}{ran}

\title[]{The interval turnpike property for adjoints}

\author[T.\ Faulwasser]{Timm Faulwasser}
\author[L.\ Gr\"une]{Lars Gr\"une}
\author[J.-P.\ Humaloja]{Jukka-Pekka Humaloja}
\author[M.\ Schaller]{Manuel Schaller}
\thanks{This work was supported by the DFG Grants GR 1569/17-1 and SCHI 1379/5-1 and was conducted while the third author was visiting University of Bayreuth under Academy of Finland Grant number 310489 held by Lassi Paunonen and supported by a travel grant from the Magnus Ehrnrooth Foundation.}

\address[T.\ Faulwasser]{TU Dortmund University, Institute of Energy Systems, Energy Efficiency and Energy Economics, Germany}
\email{{\tt timm.faulwasser@ieee.org}}
\address[L.\ Gr\"une, M.\ Schaller]{University of Bayreuth, Department of Mathematics, Germany}
\email{{\tt lars.gruene@uni-bayreuth.de,manuel.schaller@uni-bayreuth.de}}
\address[J.-P.\ Humaloja]{Tampere University, Mathematics, Computing Sciences, Finland}
\email{{\tt jukka-pekka.humaloja@tuni.fi}}

\keywords{Turnpike property, optimal control, nonlinear systems, partial differential equations}

\subjclass[2010]{93D20, 49K20, 49K40, 93B05, 93B07}

\begin{document}
	\begin{abstract}
In this work we derive an interval turnpike result for adjoints of finite- and infinite-dimensional nonlinear optimal control problems under the assumption of an interval turnpike on states and controls. We consider stabilizable dynamics governed by a generator of a semigroup with finite-dimensional unstable part satisfying a spectral decomposition condition and show the desired turnpike property under continuity assumptions on the first-order optimality conditions. We further give stronger estimates for analytic semigroups and provide a numerical example with a boundary controlled semilinear heat equation to illustrate the results.
	\end{abstract}
	\maketitle
	\section{Introduction}
	The turnpike property is a particular feature of optimal solutions of dynamic optimal control problems (OCPs). It is characterized by the phenomenon of optimal solutions to long-horizon OCPs staying close a specific steady state, the so-called turnpike, for the majority of the time. First described in a paper by von Neumann in the middle of the 20th century \cite{Neumann1945}, turnpike behavior has since received vast interest, cf.\ the recent works \cite{Gruene2018c,Gruene2019,Gugat2018,Gugat2016,Porretta2018,Sakamoto2019,Trelat18,Zaslavski2016}. For nonlinear problems, a way to derive turnpike properties is linearization of the extremal equations, analysis of the linearization and a smallness assumption, cf.\ \cite{Breiten2018,Gruene2020,Trelat2016,Trelat2015}. Another possibility is to assume a particular notion of dissipativity, cf.\ \cite{Damm2014,Faulwasser2017, Gruene2019a,Gruene2016a}, which has the advantage to allow for global turnpike properties on state and control, i.e., without a smallness condition on, e.g., the initial distance to the turnpike. In that context, however, up to now there were no results on the behavior of the corresponding adjoints. Thus, in this paper, we will show that the turnpike behavior of state and control induces turnpike behavior of the adjoints without smallness assumptions. To this end, we analyze the first-order necessary optimality conditions and, loosely speaking, show for problems governed by general evolution equations that continuity of the nonlinearities and convergence of state and control imply convergence of the adjoints. While our results are formulated in an infinite-dimensional setting, the results are new also for finite-dimensional systems, which form a special case of our setting. Besides being an important structural property of the optimal triplet, turnpike properties can be leveraged in design of numerical methods. For example \cite{Trelat2015} suggests to exploit them in indirect shooting methods, in \cite{Gruene2018c,Gruene2020} it is used for tailored discretization of infinite dimensional OCPs in a receding-horizon setting and \cite{Faulwasser2020} hinges on them in mixed-integer OCPs.
	
	After introducing the optimal control problem, the first-order conditions and the functional analytic setting in \Cref{sec:setting}, we deduce the desired turnpike property of the adjoints in \Cref{sec:stableorcontrollable} for exponentially stable and exactly controllable systems. Assuming that the underlying operator satisfies a spectrum decomposition assumption, we prove the result in \Cref{sec:finitedim} for stabilizable systems with finite-dimensional unstable part. In \Cref{sec:analytic} we provide stronger estimates assuming more structure of the problem, i.e., that the underlying semigroup is analytic and hence the dynamics are given by a parabolic evolution equation. Further in \Cref{sec:disc} we discuss and give sufficient conditions for our main assumption, i.e., interval turnpike behavior of state and control. Finally in \Cref{sec:example} we present a numerical example with a boundary controlled semilinear heat equation on a two-dimensional domain.

	\section{Setting and preliminaries}
	\label{sec:setting}
	Let $(X,\|\cdot\|_X)$ be a Banach space and $(U,\langle \cdot,\cdot\rangle_U )$ be a Hilbert space with corresponding norm $\|\cdot\|_U$. Consider the optimal control problem
	\begin{align}
	\label{DOCP}
	\begin{split}
	\min_{u\in L_2(0,T;U)} \,\, &\int_0^T J(x(t),u(t))\,dt\\
	\text{s.t.}\quad \dot{x}(t)&=\mathcal{A}x(t) + \mathcal{B}u(t)+ f(x(t),u(t))\\
	x(0)&=x_0,
	\end{split}
	\end{align}
	where $J:X\times U \to \mathbb{R}$ is sufficiently smooth, $x_0\in X$, $\mathcal{A}:D(\mathcal{A})\subset X\to X$ generates a strongly continuous semigroup on $X$, $\mathcal{B}\in L(U,X)$, and $f:X\times U\to X$ is a sufficiently smooth, locally Lipschitz nonlinearity. We will assume that the above problem has at least one optimal solution $(x,u)\in C(0,T;X)\times L_2(0,T;U)$, cf.\ \cite[Chap.\ 3]{Li1995}. Additionally, we consider $(\bar{x},\bar{u}) \in X\times U$ to be an optimal solution of the corresponding steady state system, i.e., $(\bar{x},\bar{u})$ solves
	\begin{align}
	\label{SOCP}
	\begin{split}
	\min_{u\in U}& \,\, J(x,u)\\
	\text{s.t.}\quad 0&=\mathcal{A}x+\mathcal{B}u+f(x,u),
	\end{split}
	\end{align}
Our goal in this paper is to find conditions under which interval turnpike behavior of the states and control inputs implies interval turnpike behavior of the adjoints. Our basic assumption on the behavior of the optimal solutions is thus the following.
	\begin{assumption}[Interval turnpike property for states and controls]
		\label{as:tp}
		We assume there are strictly mono\-tonous\-ly increasing functions $t_1,t_2:\mathbb{R}^{\geq 0}\to \mathbb{R}^{\geq 0}$ with $0\leq t_1(T)\leq t_2(T)\leq T $ such that $\nu(T):=t_2(T)-t_1(T)$ is strictly monotonously increasing and unbounded and for each $\varepsilon>0$ there is $T_0>0$ such that 
		\begin{align*}
		\|x(t)-\bar{x}\|_X + \|u(t)-\bar{u}\|_U\leq \varepsilon \qquad \forall t\in [t_1(T),t_2(T)],\, T\ge T_0.
		\end{align*}
	\end{assumption}
Note that this bound immediately implies $u\in L_\infty(t_1(T),t_2(T);U)$.
	\begin{remark}
		\label{rem:exp_tp}
		We say that $(x,u)\in C(0,T;X)\times L_2(0,T;U)$ satisfies the exponential turnpike property, if there is a constant $c>0$ and a decay parameter $\mu > 0$, both independent of $T$ such that we have
		\begin{align*}
		\|x(t)-\bar{x}\|_X + \|u(t)-\bar{u}\|_U \leq c\left(e^{-\mu t}+e^{-\mu(T-t)}\right).
		\end{align*}
		If this inequality holds, it can be easily seen, cf.\ \cite[Discussion after Rem.\ 6.3]{Gruene2019} that we can choose $\delta \in (0,\frac{1}{2})$ such that for each $\varepsilon>0$ there is a horizon $T$ such that
		\begin{align*}
		\|x(t)-\bar{x}\|_{L_2(\delta T,(1-\delta)T;X)} + \|u(t)-\bar{u}\|_{L_2(\delta T,(1-\delta)T;U)} \leq \varepsilon
		\end{align*}
		and
		\begin{align*}
		\|x(t)-\bar{x}\|_{C(\delta T,(1-\delta)T;X)} + \|u(t)-\bar{u}\|_{L_\infty(\delta T,(1-\delta)T;U)} \leq \varepsilon,
		\end{align*}
		i.e., $L_2$ and uniform convergence on a fixed part of the time interval $[0,T]$ for $T\to \infty$. Thus, \Cref{as:tp} is satisfied with $t_1(T)=\delta T$, $t_2(T)=(1-\delta)T$ and $\nu(T) = (1-2\delta)T$.
	\end{remark}
	The corresponding necessary optimality conditions of above problem \eqref{DOCP} read, cf.\ \cite[Chap.\ 4]{Li1995},
	\begin{align}
	\nonumber
	\dot{\lambda}(t)&=-(\mathcal{A}+f_x(x(t),u(t)))^*\lambda(t) + J_x(x(t),u(t))\\
	\label{eq:dynamics:optcond}
	0&=\left(\mathcal{B}+f_u(x(t),u(t))\right)^*\lambda(t)+J_u(x(t),u(t))\\
	\nonumber
	\dot{x}(t) &=\mathcal{A}x(t)+\mathcal{B}u(t)+f(x(t),u(t)),
\intertext{where $x(0)=x_0$ and $\lambda(T)=0$. Analogously, the optimality conditions of the steady state problem read }
	\nonumber
	0&=-(\mathcal{A}+f_x(\bar{x},\bar{u}))^*\bar{\lambda} + J_x(\bar{x},\bar{u})\\
	\nonumber
	0&=\left(\mathcal{B}+f_u(\bar{x},\bar{u})\right)^*\bar{\lambda}+J_u(\bar{x},\bar{u})\\
		\nonumber
	0&= \mathcal{A}\bar{x}+\mathcal{B}\bar{u}+f(\bar{x},\bar{u}).
	\end{align}
	Our goal in this paper is to show the interval turnpike property of the adjoint $\lambda$.
	\begin{definition}[Interval turnpike property for adjoint states]
		\label{def:adtp}
		We say that the adjoint $\lambda$ satisfies the interval turnpike property at the steady state adjoint $\bar \lambda$, if there are strictly mono\-tonous\-ly increasing functions $s_1,s_2:\mathbb{R}^{\geq 0}\to \mathbb{R}^{\geq 0}$ with $0\leq s_1(T)\leq s_2(T)\leq T $ such that $\theta(T):=s_2(T)-s_1(T)$ is strictly monotonously increasing and unbounded and for each $\varepsilon>0$ there is $T_0>0$ such that 
		\begin{align*}
		\|\lambda(t)-\bar{\lambda}\|_Y \leq \varepsilon \qquad \forall t\in [s_1(T),s_2(T)],\, T\ge T_0.
		\end{align*}%
\end{definition}

For our analysis, we define the remainder terms
	\begin{align}
\nonumber
	r_{f}(t) &:= f(x(t),u(t)) - f(\bar{x},\bar{u})
\intertext{and for $\star\in\{x,u\}$}
\nonumber
	r_{f_\star}(t) &:= f_\star(x(t),u(t)) - f_\star(\bar{x},\bar{u}),\\
	\nonumber
	r_{J_\star}(t) &:=J_\star(x(t),u(t)) - J_\star(\bar{x},\bar{u}).
\intertext{Thus, denoting $A:=\mathcal{A}+f_x(\bar x, \bar u)$, $B:=\mathcal{B}+f_u(\bar x, \bar u)$, and $(\delta x,\delta u, \delta\lambda):=(x-\bar x, u-\bar u, \lambda-\bar\lambda)$, we have that}
	\label{eq:extremal_delta}
	\dot{\delta \lambda}(t)&=-A^*\delta \lambda(t) - r_{f_x}(t)^*\delta(t) \lambda + r_{J_x}(t)\\
	\label{eq:eq_delta}
	0&=B^*\delta \lambda(t) + r_{f_u}(t)^*\delta \lambda(t) + r_{J_u}(t)\\
	\dot{\delta x}(t) &= \mathcal{A}\delta x + \mathcal{B}\delta u+r_{f}(t)\label{eq:delta_inival}
	\end{align}
	with $\delta x(0)=x_0-\bar{x}$ and $\delta \lambda(T)=-\bar{\lambda}$. 
	It is clear that the solutions $x$, $u$, and $\lambda$ of \cref{eq:dynamics:optcond} depend on $T$ and hence also $\delta x$, $\delta u$, and $\delta \lambda$ do. However, for the sake of readability, we do not explicitly indicate this dependence. We note that, using the definition of $\delta x$, $\delta u$, and $\delta \lambda$, the inequalities from Assumption \ref{as:tp} and Definition \ref{def:adtp} can be written as $\|\delta x(t)\|_X+\|\delta u(t)\|_U\le \varepsilon$ and $\|\delta\lambda(t)\|_Y\le \varepsilon$, respectively.
	\begin{remark}[Extension to box constraints on the control]
			We briefly discuss how the case of box constraints on the control
			can be treated analogously to the unconstrained case, if one assumes that the turnpike lies in the interior of the constraints. To this end, in the dynamic problem \eqref{DOCP} we add the constraints
			\begin{align}
			\label{eq:dynamicconst}
			u_a\leq {u}(t) \leq u_b
			\end{align}
			for a.e.\ $t\in [0,T]$ and in the steady state problem \eqref{SOCP} we add the constraints
			\begin{align}
			\label{eq:steadyconst}
			u_a\leq \bar{u} \leq u_b
			\end{align}
			for $u_a,u_b \in U$.
			In the case of box constraints, standard assumptions and the classical methods of calculus of variations, cf., e.g., \cite[Chap.\ 5]{Troeltzsch2010}, \cite{Hinze09} or \cite{Lions71}, ensure that, besides $\lambda$, there exist two additional multipliers $0\leq \mu_a,\mu_b\in L_2(0,T;U)$ such that the stationarity condition, i.e., the second line of \eqref{eq:dynamics:optcond} becomes
			\begin{align}
			\nonumber
			&0=\left(\mathcal{B}+f_u({x}(t),{u}(t))\right)^*{\lambda(t)} + J_u({x}(t),{u}(t)) + \mu_a(t)-\mu_b(t),\\
			\label{eq:comp}
			&0=\langle \mu_a,{u}-u_a\rangle_{L_2(0,T;U)}\qquad 0=\langle \mu_b,u_b-{u}\rangle_{L_2(0,T;U)},\\
			\nonumber
			&u_a\leq {u}(t)\leq u_b,
			\end{align}
			for a.e.\ $t\in [0,T]$.
			If we assume that the constraint is not active at an optimal solution of the steady state problem \eqref{SOCP}, i.e., the inequalities in \eqref{eq:steadyconst} are strict, then \Cref{as:tp} assures that we can choose $\varepsilon>0$ such that also the solution of the dynamic OCP does not touch the constraints on $[t_1(T),t_2(T)]$, i.e., \eqref{eq:dynamicconst} is also strict on this interval. In that case, the complementarity condition \eqref{eq:comp} assures that $\mu_a(t)=\mu_b(t)=0$ for a.e.\ $t\in [t_1(T),t_2(T)]$ and the dynamics reduce to the unconstrained case \eqref{eq:dynamics:optcond} on the subinterval $[t_1(T),t_2(T)]$.
	\end{remark}
In our subsequent analysis, we will exploit that, due to \Cref{as:tp} and continuity, the remainder terms $r_f$, $r_{f_\star}$ and $r_{J_\star}$ defined above approach zero for $t\in[t_1(T),t_2(T)]$. In order to make this property rigorous in the appropriate function spaces, we introduce superposition operators. Intuitively, a superposition operator is a nonlinear map between function spaces induced by a given nonlinear function defined on, e.g., finite-dimensional spaces by superposition. We refer the interested reader to \cite[Sec.\ 4.3.3]{Troeltzsch2010} for a short introduction and \cite{Appell1990,Goldberg1992} for an in-depth treatment of these topics in Sobolev and Lebesgue spaces of abstract functions. In order to not hide the main steps behind technical details, we only consider the case of scalar nonlinearities here.
	\begin{definition}[Superposition operator]
		\label{defn:nonlin:superpos}
		Let $\Omega\subset \mathbb{R}^n$, $n\in \mathbb{N}$. Consider a mapping $\varphi:\mathbb{R}\to \mathbb{R}$. Then the mapping $\Phi$ defined by
		\begin{align*}
		\Phi(x)(\omega) = \varphi(x(\omega)) \qquad \text{for } \omega\in \Omega
		\end{align*}
		assigns to a function $x:\Omega\to \mathbb{R}$ a new function $z:S\to \mathbb{R}$ via the relation $z(\omega) = \varphi(x(\omega))$ for $\omega\in \Omega$ and is called a Nemytskij operator or superposition operator.
	\end{definition} 
	An immediate question that arises is the following: Given a function $x\in L_p(\Omega)$, which integrability does the image $\underline\Phi(x)$ have? It turns out that in case $p<\infty$, this is coupled to growth assumptions on the underlying nonlinearity. It is to be expected as, e.g., for $\varphi(x)=x^3$, the corresponding superposition operator maps $L_{3p}(\Omega)$ to $L_{p}(\Omega)$.
	\begin{proposition}
		\label{prop:nonlin:mapsto}
		Let $\varphi:\mathbb{R}\to \mathbb{R}$ be continuous. For $1\leq p,q <\infty$ let 
		\begin{align}
		\label{eq:nonlin:growth}
		|\varphi(s)| \leq c_1+c_2|s|^{\frac{p}{q}} \qquad \forall\,s\in \mathbb{R}
		\end{align}
		for constants $c_1\in \mathbb{R}$ and $c_2 \geq 0$. Then the corresponding superposition operator $\Phi$ maps $L_p(\Omega)$ into $L_q(\Omega)$. Additionally, it is continuous as a nonlinear map from $L_p(\Omega)$ to $L_q(\Omega)$, i.e., if $\|x-z\|_{L_p(\Omega)} \to 0$, we have that
		\begin{align*}
		\|\Phi(x)-\Phi(z)\|_{L_q(\Omega)} \to 0.
		\end{align*}
	\end{proposition}
	\begin{proof}
		See \cite[Thm.\ 1]{Goldberg1992} and \cite[Thm.\ 4]{Goldberg1992}.
	\end{proof}
	It can also be shown that the assumptions of \Cref{prop:nonlin:mapsto} are not only sufficient for continuity, but also necessary, cf.\ \cite[Thm.\ 3.1]{Goldberg1992}. Thus, e.g., for a cubic nonlinearity $f(x,u)=-x^3$ and assuming that the state and control approach the turnpike in some $L_p$-norm, the remainder term $r_f$ will vanish in the $L_{p/3}$-norm. 
	We now formulate the main assumption considering the continuity of the remainder terms.
	\begin{assumption}
		\label{as:cont}
		We assume that there is a real Hilbert space $(Y,\langle \cdot,\rangle_Y)$ with corresponding norm $\|\cdot\|_Y$ such that the superposition operators induced by the remainder terms $r_{f_x}(t)$ and $r_{J_x}(t)$ for any $t\in [0,T]$ are continuous from $X$ to $L(X,Y)$ and $X$ to $Y$ respectively. Additionally, we assume that the superposition operators corresponding to the remainder terms $r_{f_u}(t)$ and $r_{J_u}(t)$ for any $t\in [0,T]$ are continuous from $U$ to $L(U,Y)$ and $U$ to $U^*\simeq U$ respectively.
	\end{assumption}
	\begin{remark}
		In the finite-dimensional setting with $X=Y=\mathbb{R}^n$ and $U=\mathbb{R}^m$, $n,m\in \mathbb{N}$, the superposition operator concept is not needed and the subsequent results will hold for all Lipschitz nonlinearities.
		In the infinite-dimensional setting the assumption on continuity of the superposition operators corresponding to $r_{f_u}(t)$ and $r_{J_u}(t)$ allows, e.g., for $Y=L_2(\Omega)$ if the cost functional is quadratic in the control and the dynamics include a polynomial nonlinearity in the control, if $U$ is embedded in a regular $L_p$ space with large p. The continuity of the superposition operators corresponding to $r_{f_x}(t)$ and $r_{f_u}(t)$ can be verified if the state space $X$ is sufficiently regular and embedded into an $L_p$-space with $p$ large and the nonlinearity is, e.g., polynomial in $x$. Additionally if the superposition operator corresponding to $f_{x}(\bar{x},\bar{u})^*$ can be extended to a compact operator from the domain of $\mathcal{A}^*$ in $Y$ to $Y$ and if the semigroup generated by $\mathcal{A}^*$ is exponentially stable, the perturbed operator $A^*=\left(\mathcal{A}+f_x(\bar{x},\bar{u})\right)^*$ generates a semigroup on $Y$, cf.\ \cite[Sec.\ III.2]{Engel2000} and \Cref{sec:example}.
	\end{remark}
	We assume that $A^*=\left(\mathcal{A}+f_x(\bar{x},\bar{u})\right)^*$ generates a strongly continuous semigroup $(\mathcal{T}^*(t))_{t\geq 0}$ on $Y$, $B\in L(U,Y)$ and that $\bar{\lambda}\in Y$ and whenever we refer to a solution of \eqref{eq:extremal_delta}, we mean it in the mild sense, i.e., for the adjoint, we have the variation of constants formula, cf. \cite[Sec.\ 4.2]{Pazy83}, 
	\begin{align}
	\label{eq:mildsol}
	\delta \lambda (t) = \mathcal{T}^*(T-t)\delta \lambda(T) + \int_t^{T} \mathcal{T}^*(s-t)\left(r_{f_x}(s)^*\delta \lambda(s)+r_{J_x}(s)\right)\,ds.
	\end{align}
	\section{Stable or exactly controllable systems}
	\label{sec:stableorcontrollable}
	We first analyze two particular cases, to which we will reduce more general systems in \Cref{sec:finitedim}: On the one hand the case where $A^*$ generates an exponentially stable semigroup on $Y$ and on the other hand the case of $(A,B)$ being exactly controllable.
	\begin{theorem}[Adjoint turnpike for exponentially stable $A^*$]\label{thm:expstab}
		Let \Cref{as:cont} hold. Let $(x,u)$ satisfy the interval turnpike property of \Cref{as:tp} with the intervals $[t_1(T),t_2(T)]$ and assume that the adjoints satisfy $\rho:=\sup_{T\ge 0}\|\delta \lambda\|_{C(t_1(T),t_2(T);X)}<\infty$. Assume that $A^*$ generates an exponentially stable semigroup $(\mathcal{T}^*(t))_{t\geq 0}$ on $Y$. Then $\lambda$ satisfies the interval turnpike property from \Cref{def:adtp}.
	\end{theorem}
	\begin{proof}
		First, we write the adjoint equation, i.e., the first equation of \eqref{eq:extremal_delta}, by means of the variation of constants formula \eqref{eq:mildsol} for $t\in [t_1(T),t_2(T)]$ on $[t,t_2(T)]$
		\begin{align*}
		\delta \lambda (t) = \mathcal{T}^*(t_2(T)-t)\delta \lambda(t_2(T)) \!+ \!\int_t^{t_2(T)} \mathcal{T}^*(s-t)\left(r_{f_x}(s)^*\delta \lambda(s)\!+\!r_{J_x}(s)\right)\,ds.
		\end{align*}
		By exponential stability of the semigroup there is $M\geq 1$ and $\mu > 0$ such that $\|\mathcal{T}^*(t)\|_{L({Y},{Y})} \leq Me^{-\mu t}$ for all $t\geq 0$. This implies the existence of $c>0$ such that the estimate
		\begin{eqnarray*}
			\|\delta \lambda(t)\|_{Y} & \leq & Me^{-\mu (t_2(T)-t)}\|\delta \lambda(t_2(T))\|_{Y}\\
			&& + \;
			c\left(\|r_{f_x}\|_{C(t,t_2(T);L(X,Y))}\rho + \|r_{J_x}\|_{C(t,t_2(T);Y)}\right)
		\end{eqnarray*}
		holds for all $t\in [t_1(T),t_2(T)]$. 
		Setting $s_1(T) := t_1(T)$, $s_2(T) = (t_2(T)-t_1(T))/2$, and recalling that $t_2(T)-t_1(T)\to\infty$ as $T\to\infty$, we obtain for sufficiently large $T$ that $Me^{-\mu (t_2(T)-t)}\le 1/2$ for all $t\in[s_1(T),s_2(T)]$. This implies
		\[ \|\delta \lambda(t)\|_{Y}  \le 2c\left(\|r_{f_x}\|_{C(t,t_2(T);L(X,Y)}\rho + \|r_{J_x}\|_{C(t,t_2(T);Y)}\right)\]
		for all $t\in[s_1(T),s_2(T)]$. The assertion follows since $\|r_{f_x}\|_{C(t,t_2(T);L(X,Y)}\to 0$ and $\|r_{J_x}\|_{C(t,t_2(T);Y)}\to 0$ as $T\to\infty$ due to Assumptions \ref{as:tp} and \ref{as:cont}.
	\end{proof}
	
	\begin{remark}\label{rem:addterm} If we add a term $\sigma_T(t)$ with $\|\sigma_T\|_{C(t,t_2(T);Y)} <\infty$ on the right hand side of \eqref{eq:extremal_delta}, then a straightforward extension of the proof shows that for all sufficiently large $T$ we obtain
		\[ \|\delta \lambda(t)\|_{Y} \le \varepsilon + 2c \|\sigma_T\|_{C(t,t_2(T);{Y})} \]
		for all $t\in[s_1(T),s_2(T)]$.
	\end{remark}
	
	Next, we discuss the case of $(A,B)$ being exactly controllable.
	\begin{definition}[{Exact and approximate controllability, \cite[Def.\ 4.1.3]{Curtain1995}}]
		For any $\tau \in [0,T]$, we define the controllability map $\phi_\tau : L_2(0,\tau;U) \to {Y}$ by
		\begin{align*}
		\phi_\tau u := \int_0^\tau \mathcal{T}(\tau-s)Bu(s)\,ds.
		\end{align*}
		We call $(A,B)$ exactly controllable in time $t_c>0
		$ if $\ran \phi_{t_c}={Y}$. Similarly, we call $(A,B)$ approximately controllable in time $t_c$ if $\overline{\ran \phi_{t_c}} = {Y}$. 
		\label{def:ctrl}
	\end{definition}

	It is clear that exact and approximate controllability coincide in finite-dimensions. An important characterization of controllability is the following observability inequality, which was proven first in the seminal paper \cite{Lions1988a} with the Hilbert Uniqueness Method.
	\begin{theorem}[{\cite[Thm.\ 4.1.7]{Curtain1995}}]
		$(A,B)$ is exactly controllable in time $t_c>0$ if and only if there is $\alpha_{t_c} > 0$ such that
		\begin{align*}
		\int_0^{t_c}\|B^*\mathcal{T}^*(s)x_0\|^2_U\,ds \geq \alpha_{t_c} \|x_0\|_{Y}^2 \quad \forall\, x_0\in {Y}.
		\end{align*}
	\end{theorem}
	Using substitution in the previous estimate we immediately obtain that
	\begin{align}
	\label{eq:obs_bounded}
	\int_{t-t_c}^{t}\|B^*\mathcal{T}^*(t-s)\delta \lambda(t)\|^2_U\,ds \geq \alpha_{t_c} \|\delta \lambda(t)\|_{Y}^2 \quad \forall\,\delta \lambda(t)\in {Y},\quad t\in [t_c,T]
	\end{align} 
	
This enables us to derive the following bound on $\|\delta \lambda(t)\|_Y$.
	
	\begin{theorem}
		\label{thm:excont}
		Let $(A,B)$ be exactly controllable in time $t_c>0$. Then there is $c>0$ independent of $T$, such that 
		\begin{align*}
		\|\delta \lambda(t)\|_{Y}^2 \leq c\!\int_{t-t_c}^t\|r_{f_u}(s)^*\delta \lambda(s) + r_{J_u}(s)\|^2_U \!+\! \|\!-r_{f_x}(s)^*\delta \lambda(s)+r_{J_x}(s)\|_{Y}^2\,ds.
		\end{align*}
	\end{theorem}
	\begin{proof}
		The proof of this estimate is inspired by \cite[Proof of Rem.\ 2.1]{Porretta2013}, where the finite-dimensional case is considered. We decompose $\delta \lambda = \delta \lambda_1 + \delta \lambda_2$, where for any $s<t$ 
		\begin{align*}
		&&\delta \lambda_1'(s) &= -A^*\delta\lambda_1(s), && \delta\lambda_1(t)=\delta \lambda(t),&&\\
		&&\delta\lambda_2'(s) &= -A^*\delta\lambda_2(s)-r_{f_x}(s)^*\delta\lambda_2(s) + r_{J_x}(s), && \delta\lambda_2(t)=0,&&
		\end{align*}
		and apply the observability estimate \eqref{eq:obs_bounded} to $\delta \lambda_1(s) = \mathcal{T}^*(t-s)\delta \lambda(t)$. This yields \begin{align*}
		\alpha_{t_c}\|\delta\lambda(t)\|_{Y}^2 \leq \int_{t-t_c}^t \|B^*\delta\lambda_1(s)\|_U^2 \,ds \leq \int_{t-t_c}^t \|B^*\delta \lambda(s)\|_U^2 +\|B^*\delta \lambda_2(s)\|_U^2\,ds.
		\end{align*}
		Further, we estimate
		\begin{align*}
		\int_{t-t_c}^t \|B^*\delta \lambda_2(s)\|_U^2\,ds &\leq \int_{t-t_c}^t\|B^*\int_s^t \mathcal{T}^*(\tau-s)(-r_{f_x}(\tau)^*\delta\lambda_2(\tau)\\
		& \hspace*{5cm}  + r_{J_x}(\tau))\,d\tau\|^2_U\,ds\\ &\leq c(t_c) \int_{t-t_c}^t\|-r_{f_x}(s)^*\delta\lambda_2(s) + r_{J_x}(s)\|_{Y}^2\,ds.
		\end{align*}
		Finally, by \eqref{eq:eq_delta}, we have that 
		\begin{align*}
		\int_{t-t_c}^t\|B^*\delta\lambda(s)\|_U^2 = \int_{t-t_c}^t\|r_{f_u}(s)^*\delta \lambda(s) + r_{J_u}(s)\|^2_U\,ds,
		\end{align*}
		which concludes the proof.
	\end{proof}
	
Since the right hand side of the inequality	from Theorem \ref{thm:excont} obviously tends to zero if the integrands tend to zero, we immediately obtain the following corollary.
	
	\begin{corollary}
		\label{cor:conv_exact}
		Let \Cref{as:cont} hold and let $(A,B)$ be exactly controllable in time $t_c>0$. Let $(x,u)$ satisfy the turnpike property of \Cref{as:tp} with the intervals $[t_1(T),t_2(T)]$ and assume that the adjoints satisfy $\rho:=\sup_{T\ge 0}\|\delta \lambda\|_{C(t_1(T),t_2(T);{Y})}<\infty$. Then $\lambda$ satisfies the interval turnpike property from \Cref{def:adtp} with $s_1(T)=t_1(T)+t_c$ and $s_2(T)=t_2(T)$.
	\end{corollary}
	\begin{proof}
		Follows immediately from Theorem \ref{thm:excont}.
	\end{proof}
	\begin{remark}\label{rem:addterm2} Similar to Remark \ref{rem:addterm}, it is easily seen from the proof of \Cref{thm:excont} that if we add a term $\sigma_T(t)$ with $\|\sigma_T\|_{C(t-t_c,t;{Y})} \le \bar\sigma_T < \infty$ and $\rho_T(t)$  on the right hand sides of \eqref{eq:extremal_delta} and \eqref{eq:eq_delta}, respectively, then  the result of Theorem \ref{thm:excont} changes to
		\begin{align}
		\|\delta \lambda(t)\|_{Y}^2 \leq c\int_{t-t_c}^t  \|r_{f_u}(s)^*\delta &\lambda(s) + r_{J_u}(s)+ \sigma_T(s)\|^2_U  \\+\|&-r_{f_x}(s)^*\delta \lambda(s)\nonumber
		+r_{J_x}(s)  + \rho_T(s)\|_{Y}^2\,ds.\label{eq:addterm2}
		\end{align}\end{remark}
	We then obtain as a counterpart for the inequality in \Cref{def:adtp} the bound
	\begin{align*}
	\|\delta \lambda(t)\|_{Y} \leq \varepsilon + c(\bar \sigma_T+\bar \rho_T)\qquad \forall t\in [s_1(T),s_2(T)],\, T\ge T_0.
	\end{align*}
	
	\section{Stabilizable systems with finite-dimensional unstable part}
	\label{sec:finitedim}
In this section we extend our results to exponentially detectable $(A^*,B^*)$, where the unstable part of $A^*$ is finite-dimensional and $B^*$ has finite rank. We note that this includes all finite-dimensional systems with stabilizable linearization. In order to define the correct setting for infinite-dimensional systems, we present the spectrum decomposition assumption as follows.
	\begin{definition}[{\cite[Def.\ 5.2.5]{Curtain1995}}]
		\label{def:sda}
		Denoting $\sigma^+(A) := \sigma(A)\cap \left\{ s \in
		\mathbb{C}\!:\! \operatorname{Re}s\! \geq\! 0 \right\}$ and $\sigma^-(A)$
		= $\sigma(A)\cap \left\{ s \in \mathbb{C}\!:\! \operatorname{Re}s\! <\!
		0 \right\}$, an operator $A$ satisfies the spectral decomposition assumption
		if $\sigma^+(A)$ is bounded and separated from $\sigma^-(A)$ in such
		a way that a rectifiable, simple, closed curve $\Gamma$ can be drawn
		so as to enclose an open set containing $\sigma^+(A)$ in its
		interior and $\sigma^-(A)$ in its exterior.
	\end{definition}
	
	\begin{remark} \label{rem:sda}
		Classes of operators satisfying the spectrum decomposition assumption
		include, e.g., delay equations \cite[Sec.\ 2.4]{Curtain1995} and
		Riesz-spectral operators with a pure point spectrum and only finitely
		many eigenvalues in $\sigma^+(A)$. More concrete examples of the latter
		are compact perturbation of the Laplace operator, i.e., $A = \Delta +
		c^2I$ for $c \in \mathbb{R}$ or models of damped vibrations such as 
		$$
		A = 
		\begin{bmatrix}
		0 & I \\ -A_0 & -D
		\end{bmatrix}
		$$
		where $A_0$ is a positive operator and $D$ is an unbounded damping operator
		(see, e.g.,\cite{JacTru08} and the Euler-Bernoulli example with
		Kelvin-Voigt damping).
	\end{remark}
	
	If $A^*$ satisfies the decomposition assumption, by
	\cite[Lem.\ 2.5.7]{Curtain1995} the decomposition of the spectrum
	induces a corresponding decomposition of ${Y}$. Defining the spectral
	projection $P$ by 
	$$
	Py_0 := \frac{1}{2\pi i} \int\limits_{\Gamma} (sI - A^*)^{-1}y_0\,ds
	$$
	for $y_0\in Y$, where $\Gamma$ from \Cref{def:sda} is traversed once in the
	positive direction, we obtain the decomposition ${Y} = {Y}_u \oplus {Y}_s$,
	where ${Y}_u = P{Y}$ and ${Y}_s = (I-P){Y}$.  Moreover, the spectral projection yields a linear coordinate transform such that the pair $(A^*,B^*)$ can be transformed into the form
	\begin{equation}
	\label{eq:sd}
	\widetilde{A^*} = 
	\begin{bmatrix}
	A^*_u & 0 \\ 0 & A^*_s
	\end{bmatrix},  \qquad
	\widetilde{B^*} = 
	\begin{bmatrix}
	B_u^* & B_s^*
	\end{bmatrix}
	\end{equation}
	where $A^*_u, B^*_u, A^*_s, B^*_s$ are restrictions of $A^*$ and $B^*$ to ${Y}_u$ and ${Y}_s$, respectively. Note that $A^*_u$ and $B^*_u$ are bounded operators. We impose the following assumption on $A^*$.
	\begin{assumption} \label{as:sda}
		$A^*$ satisfies the spectrum decomposition assumption such that it has
		the decomposition according to \eqref{eq:sd}, where $A^*_u$ is
		finite-di\-men\-sio\-nal and $A^*_s$ is exponentially stable.
	\end{assumption}
	
	If we split up the transformed adjoint accordingly via 
	\begin{equation}\label{eq:adjtrafo} \widetilde{\delta \lambda} = \left(\begin{array}{c} \delta\lambda_u \\ \delta\lambda_s \end{array}\right),\end{equation}
	then the equations \eqref{eq:extremal_delta} and \eqref{eq:eq_delta} attain the form
	\begin{align}
	\dot{\delta\lambda_u} & = -A_u^*\delta\lambda_u - \tilde r_1^*\delta\lambda_u - \tilde r_2^*\delta\lambda_s + \tilde r_3\label{eq:adj1}\\
	\dot{\delta\lambda_s} & = -A_s^*\delta\lambda_s - \tilde r_7^*\delta\lambda_u - \tilde r_8^*\delta\lambda_s + \tilde r_9\label{eq:adj3}\\
	0 & = B^*_u\delta\lambda_u + B^*_s\delta\lambda_s + \tilde r_4\delta\lambda_u + \tilde r_5 \delta\lambda_s + \tilde r_6. \label{eq:adj4}
	\end{align}
	Here, the terms $\tilde r_j$ are derived via coordinate transformation and splitting from the remainder terms in \eqref{eq:extremal_delta}--\eqref{eq:delta_inival} and---up to multiplication by appropriate constants---satisfy the same estimates as these remainder terms.
	Using this decomposition, we can prove the following theorem.
	
	\begin{theorem}[$(A^*,B^*)$ exponentially detectable, $B^*$ has finite
		rank]\label{thm:parstab} Let \Cref{as:cont} hold. Let $(x,u)$
		satisfy the turnpike property of \Cref{as:tp} on
		$[t_1(T),t_2(T)]$ and assume that $\rho:=\sup_{T\ge 0}\|\delta
		\lambda\|_{C(t_1(T),t_2(T);{Y})}<\infty$. Let \Cref{as:sda} hold and
		further assume that $B^*$ has finite rank and $(A^*, B^*)$ is exponentially
		detectable. 
Then $\lambda$ satisfies the interval turnpike property from \Cref{def:adtp}.
	\end{theorem}
	\begin{proof}
		First note that the claimed property holds for $\delta\lambda$ if and only if it holds for the transformed adjoint $\widetilde{\delta\lambda}$ from \eqref{eq:adjtrafo}. The property for $\widetilde{\delta\lambda}$, in turn, holds if and only if it holds for the two components $\delta\lambda_u$ and $\delta\lambda_s$. Moreover, note that the assumed exponential detectability of $(A^*, B^*)$ implies that the finite-dimensional pair $(A_u^*, B_u^*)$ is (exponentially) detectable, and hence, (exactly) observable by Hautus rank condition \cite[Def.\ 1.2.6]{Curtain1995}.
		
		We start by applying the extension of Theorem \ref{thm:expstab} described in Remark \ref{rem:addterm} to $\delta\lambda_s$, with $\sigma_T = - \tilde r^*_7 \delta\lambda_u$. We note that the fact that equation \eqref{eq:adj4} contains additional terms compared to equation \eqref{eq:eq_delta} does not affect the applicability of Theorem \ref{thm:expstab} and Remark \ref{rem:addterm}, because equation \eqref{eq:eq_delta} is not used in its proof. Due to the uniform boundedness of $\delta\lambda$ which implies uniform boundedness of $\delta\lambda_u$, $\sigma_T$ tends to zero as $T\to\infty$ on $[t_1(T),t_2(T)]$. Hence, we obtain the desired property for $\delta\lambda_s$ on an interval $[\tilde s_1(T), \tilde s_2(T)]$. We note that by the construction in the proof of Theorem \ref{thm:expstab} we obtain $[\tilde s_1(T), \tilde s_2(T)]\subset[t_1(T),t_2(T)]$.
		
		Now for $\delta\lambda_u$ we use the extension of Theorem \ref{thm:excont} described in Remark \ref{rem:addterm2} with $\sigma_T = -\tilde r_2^*\delta\lambda_s$ and $\rho_T = B_s^*\delta\lambda_s-\tilde r_5^*\delta\lambda_s$, on $[\tilde s_1(T), \tilde s_2(T)]$. Since all terms become arbitrarily small on $[\tilde s_1(T), \tilde s_2(T)]$ as $T\to\infty$, we obtain the desired estimate for $\delta\lambda_u$ on $[s_1(T),s_2(T)]$ with $s_1(T)=\tilde s_1(T) + t_c$ and $s_2(T)=\tilde s_2(T)$.
	\end{proof}
	
	\begin{remark}
		Note that if $B^*$ has finite rank and \Cref{as:sda} is satisfied,
		$(A^*,B^*)$ is exponentially detectable in particular if $(A^*,B^*)$ is
		approximately observable: in that case both
		$(A_s^*,B_s^*)$ and $(A_u^*, B_u^*)$ are approximately observable \cite[Lem.\ 9.7.2]{StaBook}, which
		for the finite-dimensional pair $(A_u^*,B_u^*)$ coincides with exact
		observability. 
	\end{remark}

\begin{remark} For linear quadratic problems, detectability and stabilizability imply an exponential turnpike property and thus an interval turnpike property for states, controls and adjoints for a very general class of infinite dimensional systems \cite[Thm.\ 17]{Gruene2019}.
\end{remark}
 
	\begin{remark} The assumptions on $(A^*,B^*)$ in \Cref{thm:parstab} are in particular satisfied if $(A^*,B^*)$ is finite-dimensional and detectable. In that case, the result can be alternatively proven by
		using the decomposition into the observable and nonobservable subspaces of $(A^*,B^*)$, which is of the form
		$$
		\widetilde {A^*} = \begin{bmatrix}
		A^*_1 & 0 \\ A_2^* & A_3^*
		\end{bmatrix}, \qquad
		\widetilde {B^*} = \begin{bmatrix}
		B_1^* & 0 
		\end{bmatrix},
		$$
		where $(A^*_1,B^*_1)$ is (exactly) observable and $A^*_3$ is (exponentially) stable (the unstable subspace is contained in the observable subspace by the Hautus rank condition). Then the equations for $\widetilde{\delta\lambda}$ become 
		\begin{align}
		\dot{\delta\lambda_1} & = -A_1^*\delta\lambda_1 - \tilde r_1^*\delta\lambda_1 - \tilde r_2^*\delta\lambda_2 + \tilde r_3\\
		\dot{\delta\lambda_2} & = -A_2^*\delta\lambda_1 - A_3^*\delta\lambda_2 - \tilde r_7^*\delta\lambda_1 - \tilde r_8^*\delta\lambda_2 + \tilde r_9\\
		0 & = B^*_1\delta\lambda_1 + \tilde r_4\delta\lambda_1 + \tilde r_5 \delta\lambda_2 + \tilde r_6.
		\end{align}
		For proving Theorem \ref{thm:parstab} we can then proceed similarly as above, but in reverse order: We first obtain the desired estimate for $\delta\lambda_1$ using Theorem \ref{thm:excont} and Remark \ref{rem:addterm2} and then obtain the estimate for $\delta\lambda_2$ using Theorem \ref{thm:expstab}, Remark \ref{rem:addterm} and the estimate for $\delta\lambda_1$ from the first part of the proof.
		
		Note that the observability decomposition has very limited use for infinite-dimensional systems as the nonobservable subspace is the largest $T^*(t)$-invariant subspace in ${Y}$ contained in $\ker B^*$ \cite[Lem.\ 4.1.18]{Curtain1995}. For example, if $(A^*,B^*)$ is approximately observable, then the nonobservable subspace is the trivial subspace $\{0\}$ and the observability decomposition is redundant.
	\end{remark}
%
		
	\section{Particular case of an analytic semigroup}
	\label{sec:analytic}
	We briefly present a refined version of \Cref{thm:expstab} for the case where $A^*$ generates an analytic semigroup. In this case, we can improve the estimate in \Cref{def:adtp} by using stronger norms. To this end, we define the space 
	\begin{align*}
	W^{A^*}(0,T):= \left\{v\in L_2(0,T;{Y})\,:\,v'\in L_2(0,T;{Y})\right\}\cap L_2(0,T;D(A^*)).
	\end{align*}
	endowed with the norm
	\begin{align*}
	\|v\|_{W^{A^*}(0,T)}	:= \|v'\|^2_{L_2(0,T,{Y})} + \|v\|^2_{L_2(0,T;{Y})}+\|A^*v\|^2_{L_2(0,T;{Y})}.
	\end{align*}
	It was shown in, e.g., \cite[Pt.\ II-1, Rem.\ 4.2]{Bensoussan2007} that 
	\begin{align}
	\label{eq:embedding}
	W^{A^*}(0,T)\hookrightarrow C(0,T;(D(A^*),{Y})_{\frac{1}{2}}),
	\end{align} with embedding constant $c_E>0$ independently of $T$, where $(D(A^*),{Y})_{\frac{1}{2}}$ denotes the real interpolation space as defined in \cite[Pt.\ II-1, Sec.\ 4.3]{Bensoussan2007}. If $A^*$ generates an exponentially stable semigroup, this interpolation space can be shown to be isomorphic to the domain of the fractional power $(A^*)^{\frac{1}{2}}$ in many applications, cf.\ \cite[Sec.\ 0.2.1]{Lasiecka2000}.
	\begin{theorem}
		\label{thm:analytic}
		Let the assumptions of \Cref{thm:expstab} hold. Assume additionally that the semigroup $(\mathcal{T}^*(t))_{t\geq 0}$ is analytic. Suppose that for each $\tilde{\varepsilon}>0$ there is $\tilde{T}_0>0$ such that
		\begin{align*}
		\int_{t_1(T)}^{t_2(T)}\|r_{f_x}(s)\|_{L({X},{Y})} + \|r_{J_x}(s)\|_{Y}\,ds \leq \tilde{\varepsilon} \qquad \forall T\geq \tilde{T}_0.
		\end{align*} Then the interval turnpike property for adjoints from \Cref{def:adtp} holds and, in addition, for each $\varepsilon>0$ there is $T_0>0$ such that the adjoints satisfy
				\begin{align*}
		\|\delta \lambda\|_{W^{A^*}(t,t_2(T))} +\|\delta \lambda(t)\|_{(D(A^*),{Y})_{\frac{1}{2}}}\leq \varepsilon \qquad \forall t\in [s_1(T),s_2(T)],\,T\geq T_0.
		\end{align*}
	\end{theorem}
	\begin{proof}
		By exponential stability with decay rate $\mu > 0$, the scaled semigroup\\ \mbox{$(e^{\frac{\mu}{2}(t_2(T)-t)}\mathcal{T}^*(t_2(T)-t))_{t\geq 0}$} is still exponentially stable and its generator has the same domain as $A^*$, cf.\ \cite[Chap.\ III, Thm.\ 1.3]{Engel2000}.
		Hence we can apply a well-known estimate for exponentially stable analytic semigroups, cf.\ \cite[Pt.\ II-1, Prop.\ 3.7, Thm.\ 3.1]{Bensoussan2007} and for $\delta \lambda$ solving \eqref{eq:extremal_delta} we obtain that
		\begin{align*}
		&\|\delta \lambda\|_{W^{A^*}(t,t_2(T))} \\&\leq c \left(\|-r_{f_x}^*\delta \lambda + r_{J_x}\|_{L_2(t,t_2(T);{Y})} + Me^{-\frac{\mu}{2}(t_2(T)-t)}\|\delta \lambda(t_2(T))\|_{(D(A^*),{Y})_{\frac{1}{2}}}\right).
		\end{align*}
		Using the boundedness of $\delta \lambda$ and the integral convergence of $r_{f_x}$ and $r_{J_x}$, the first term approaches zero for $T \to \infty$. For the second term, we set $s_1(T) := t_1(T)$ and $s_2(T) = (t_2(T)-t_1(T))/2$ and recalling that $t_2(T)-t_1(T)\to\infty$ as $T\to\infty$, there is $T$ such that $Me^{-\frac{\mu}{2} (t_2(T)-t)}\le \frac{c_E}{2c}$ for all $t\in[s_1(T),s_2(T)]$. Thus we can estimate
		\begin{align*}
		Me^{-\frac{\mu}{2}(t_2(T)-t)}\|\delta \lambda(t_2(T))\|_{(D(A^*),{Y})_{\frac{1}{2}}} \leq \frac{c_E}{2c}\|\delta \lambda(t_2(T)))\|_{(D(A^*),{Y})_{\frac{1}{2}}}.
		\end{align*}
		Using the embedding \eqref{eq:embedding} with $T$-independent embedding constant $c_E$ and subtracting  $\frac{c_E}{2}\|\delta \lambda(t_2(T)))\|_{(D(A^*),{Y})_{\frac{1}{2}}}$ on both sides yields the result.
	\end{proof}
	\begin{remark}
		The assumption of the $L_2$-convergence of the remainder terms in \Cref{thm:analytic} is satisfied if the convergence of state and control to the turnpike is, e.g., exponential as in \Cref{rem:exp_tp}.
	\end{remark}
	\section{Discussion of assumptions}
	\label{sec:disc}
	In this part we will briefly give sufficient conditions to render the interval turnpike of the primal variables as our main assumption \Cref{as:tp} fulfilled. These conditions can be guaranteed by a combination of strict dissipativity, controllability, and stabilizability. We further provide an example where the boundedness of the adjoint $\sup_{T\ge 0}\|\lambda-\bar{\lambda}\|_{C(t_2(T),T;H^1(\Omega))} < \infty$ assumed in \Cref{thm:expstab}, \Cref{thm:parstab}, and \Cref{cor:conv_exact} holds. 
First, we give a theorem stating that under suitable assumptions, measure turnpike implies interval turnpike.
\begin{theorem} Assume that the following conditions hold:
	\begin{enumerate}
		\item[(i)] The system has the measure turnpike property, i.e., for each $\varepsilon>0$ the Lebesgue measure of the set of times $t\in[0,T]$ for which
		\[ \|x(t)-\bar{x}\|_X + \|u(t)-\bar{u}\|_U > \varepsilon \]
		holds is bounded independent of $T$.
		\item[(ii)] The system is stabilizable at $\bar x$ with cost proportional to the initial distance to $\bar x$, i.e., there exists a constant $C>0$, a neighborhood $N_1$ of $\bar x$ and a function $\eta\in\mathcal{K}$ 
		such that for all $x_0\in N_1$ and each $T>0$ there is $u\in L_2(0,T;U)$ with $x(T)\in B_{C\|x_0-\bar x\|}(\bar x)$ and $\int_0^T J(x,u) dt \le TJ(\bar x, \bar u) + \eta(\|x_0-\bar x\|)$.
		\item[(iii)] The optimal value function 
		\[ V_T(x_0) := \min_{u\in L_2(0,T;U)} \int_0^T J(x,u) dt \]
		is approximately continuous in $\bar x$ uniformly in $T$ in the following sense: there are $\gamma\in\mathcal{K}$, $\rho\in\mathcal{L}$ and a neighborhood $N_2$ of $\bar x$ such that 
		\[ |V_T(x) - V_T(\bar x)| \le \gamma(\|x-\bar x\|_X) + \rho(T) \]
		for all $x\in N_2$ and all $T\ge 0$.
		\item[(iv)] Excursions from the turnpike are more expensive than staying in the turnpike, i.e., there is a function $\sigma\in\mathcal{K}_\infty$ and a neighborhood $N_3$ of $\bar x$ such that for each $T>0$ and each trajectory $x$ and control $u$ satisfying $x(0), x(T)\in N_3$ the inequality
		\begin{eqnarray*} \int_0^T J(x,u)dt & \ge & TJ(\bar x,\bar u) - \sigma(\|x(0)-\bar x\|_X) - \sigma(\|x(T)-\bar x\|_X)\\
			&& + \;\; \max_{t\in[0,T]} \sigma(\|x(t)-\bar x\|_X) \end{eqnarray*}
		holds. 
	\end{enumerate}
	Then the interval turnpike property from Assumption \ref{as:tp} holds.
	\label{eq:inttp}
\end{theorem}
\begin{proof} For each $\tilde\varepsilon>0$ that is sufficiently small such that $B_{\tilde\varepsilon}(\bar x)\subset N_1$ and $B_{C\tilde\varepsilon}(\bar x)\subset N_2\cap N_3$ denote by $\mu(\tilde\varepsilon)$ the bound on the Lebesgue measure from (i) for $\varepsilon=\tilde\varepsilon$. Then, for any $T\ge 8\mu(\tilde\varepsilon)$ there exist times $\hat t_1(T,\tilde\varepsilon)\in [T/8,3T/8]$ and $\hat t_2(T,\tilde\varepsilon)\in [5T/8,7T/8]$ such that
	\[ \|x(t)-\bar{x}\|_X + \|u(t)-\bar{u}\|_U \le \tilde\varepsilon \]
	for $t=\hat t_1(T,\tilde\varepsilon)$ and $t=\hat t_2(T,\tilde\varepsilon)$. The overall cost of the optimal solution can be written as 
	\[ \int_0^T J(x,u)dt = \int_0^{\hat t_1(T,\tilde \varepsilon)}J(x,u)dt  + \int_{\hat t_1(T,\tilde \varepsilon)}^{\hat t_2(T,\tilde \varepsilon)}J(x,u)dt + \int_{\hat t_2(T,\tilde \varepsilon)}^{T}J(x,u)dt, \]
	where we omit the argument ``$(t)$'' in the integrands in order to shorten the notation.
	In addition, we consider the control $\hat u$ and corresponding trajectory $\hat x$ that is constructed as follows: it follows the optimal control until time $\hat t_1(T,\tilde \varepsilon)$, then uses the control from (ii) with $x_0=\hat x(\hat t_1(T,\tilde \varepsilon))$ from $\hat t_1(T,\tilde \varepsilon)$ until $\hat t_2(T,\tilde \varepsilon)$, and finally it uses the optimal control for horizon $T-\hat t_2(T,\tilde \varepsilon)$ and initial condition $\hat t_2(T,\tilde \varepsilon)$. The overall cost of this trajectory can be decomposed in the same way 
	\[ \int_0^T J(\hat x, \hat u)dt = \int_0^{\hat t_1(T,\tilde \varepsilon)}J(\hat x,\hat u)dt  + \int_{\hat t_1(T,\tilde \varepsilon)}^{\hat t_2(T,\tilde \varepsilon)}J(\hat x,\hat u)dt + \int_{\hat t_2(T,\tilde \varepsilon)}^{T}J(\hat x,\hat u)dt. \]
	Now, by construction of $\hat x$ and $\hat u$ as well as (ii) and (iii) and the fact that tails of optimal trajectories are optimal trajectories, we get
	\begin{equation} \int_0^{\hat t_1(T,\tilde \varepsilon)} J(\hat x, \hat u)dt = \int_0^{\hat t_1(T,\tilde \varepsilon)} J(x, u)dt,\label{eq:p1}\end{equation}
	\begin{equation} \int_{\hat t_1(T,\tilde \varepsilon)}^{\hat t_2(T,\tilde \varepsilon)}J(\hat x,\hat u)dt \le (\hat t_2(T,\tilde \varepsilon)-\hat t_1(T,\tilde \varepsilon))J(\bar x, \bar u) + \eta(\tilde\varepsilon) \label{eq:p2}\end{equation}
	and
	\begin{equation}  \left| \int_{\hat t_2(T,\tilde \varepsilon)}^{T}J(\hat x,\hat u)dt - \int_{\hat t_2(T,\tilde \varepsilon)}^{T}J(x, u)dt\right| \le \gamma(C\tilde\varepsilon) + \gamma(\tilde\varepsilon) + \rho(T-\hat t_2(T,\tilde \varepsilon)).\label{eq:p3}\end{equation}
	Since $x$ and $u$ are optimal, we moreover obtain
\begin{eqnarray*} 
	&& \int_0^{\hat t_1(T,\tilde \varepsilon)}J(x,u)dt  + \int_{\hat t_1(T,\tilde \varepsilon)}^{\hat t_2(T,\tilde \varepsilon)}J(x,u)dt + \int_{\hat t_2(T,\tilde \varepsilon)}^{T}J(x,u)dt\\
	&& \le \;\; \int_0^{\hat t_1(T,\tilde \varepsilon)}J(\hat x,\hat u)dt  + \int_{\hat t_1(T,\tilde \varepsilon)}^{\hat t_2(T,\tilde \varepsilon)}J(\hat x,\hat u)dt + \int_{\hat t_2(T,\tilde \varepsilon)}^{T}J(\hat x,\hat u)dt,\end{eqnarray*}
	implying
	\begin{eqnarray*} 
\int_{\hat t_1(T,\tilde \varepsilon)}^{\hat t_2(T,\tilde \varepsilon)}J(x,u)dt
& \le & \int_0^{\hat t_1(T,\tilde \varepsilon)}J(\hat x,\hat u)dt - \int_0^{\hat t_1(T,\tilde \varepsilon)}J(x,u)dt\\
&&   + \; \int_{\hat t_1(T,\tilde \varepsilon)}^{\hat t_2(T,\tilde \varepsilon)}J(\hat x,\hat u)dt\\
&& + \; \int_{\hat t_2(T,\tilde \varepsilon)}^{T}J(\hat x,\hat u)dt - \int_{\hat t_2(T,\tilde \varepsilon)}^{T}J(x,u)dt.
\end{eqnarray*}
	Inserting \eqref{eq:p1}--\eqref{eq:p3} into this inequality, we arrive at
	\begin{eqnarray*} && \int_{\hat t_1(T,\tilde \varepsilon)}^{\hat t_2(T,\tilde \varepsilon)}J(x,u)dt\\
		&& \le \;\; (\hat t_2(T,\tilde \varepsilon)-\hat t_1(T,\tilde \varepsilon))J(\bar x, \bar u) + \eta(\tilde\varepsilon) + \gamma(C\tilde\varepsilon) + \gamma(\tilde\varepsilon) + \rho(T-\hat t_2(T,\tilde \varepsilon))\\
		&& \le \;\; (\hat t_2(T,\tilde \varepsilon)-\hat t_1(T,\tilde \varepsilon))J(\bar x, \bar u) + \eta(\tilde\varepsilon) + \gamma(C\tilde\varepsilon) + \gamma(\tilde\varepsilon) + \rho(T/8).
	\end{eqnarray*}
	From (iv) we thus obtain
	\begin{eqnarray*}
		&&(\hat t_2(T,\tilde \varepsilon)-\hat t_1(T,\tilde \varepsilon))J(\bar x, \bar u) + \eta(\tilde\varepsilon) + \gamma(C\tilde\varepsilon) + \gamma(\tilde\varepsilon) + \rho(T/8)\\
		&& \ge \;\; (\hat t_2(T,\tilde \varepsilon)-\hat t_1(T,\tilde \varepsilon))J(\bar x,\bar u) - \sigma(\|x(\hat t_1(T,\tilde \varepsilon))-\bar x\|_X)\\
		&& \qquad - \;\; \sigma(\|x(\hat t_2(T,\tilde \varepsilon))-\bar x\|_X)+ \max_{t\in[\hat t_1(T,\tilde \varepsilon),\hat t_2(T,\tilde \varepsilon)]} \sigma(\|x(t)-\bar x\|_X)\\
		&& \ge \;\; (\hat t_2(T,\tilde \varepsilon)-\hat t_1(T,\tilde \varepsilon))J(\bar x,\bar u) - \sigma(\tilde \varepsilon)\\
		&& \qquad - \;\; \sigma(C\tilde \varepsilon)+ \max_{t\in[\hat t_1(T,\tilde \varepsilon),\hat t_2(T,\tilde \varepsilon)]} \sigma(\|x(t)-\bar x\|_X).
	\end{eqnarray*}
	Using the bounds on $\hat t_1(T,\tilde \varepsilon)$ and $\hat t_2(T,\tilde \varepsilon)$ this implies 
	\[ \max_{t\in[3T/8,5T/8]} \|\hat x(t)-\bar x\|_X \le \sigma^{-1}(\eta(\tilde\varepsilon) +  \gamma(\tilde\varepsilon) + \sigma(\tilde\varepsilon) + \gamma(C\tilde\varepsilon) + \sigma(C\tilde\varepsilon) + \rho(T/8)).
	\]
	Now, defining $t_1(T)=3T/8$ and $t_2(T)=5T/8$, the property from Assumption \ref{as:tp} follows for every $\varepsilon > 0$ by setting $\tilde \varepsilon >0$ and $T\ge 8\mu(\varepsilon)$ so large that $\sigma^{-1}(\eta(\tilde\varepsilon) +  \gamma(\tilde\varepsilon) + \sigma(\tilde\varepsilon) + \gamma(C\tilde\varepsilon) + \sigma(C\tilde\varepsilon) + \rho(T/8))<\sigma$. 
\end{proof}
\begin{remark}
	\label{rem:dissipativity} The properties needed in the assumption of Theorem \ref{eq:inttp} follow from other well known properties of the optimal control problem under consideration:
	\begin{itemize}
		\item Condition (i) follows from strict dissipativity and Condition (ii). Strict dissipativity demands the existence of a storage function $S:X\to\mathbb{R}$, bounded from below, and a function $\alpha\in\mathcal{K}_\infty$ such that
		\[ S(x(t)) \le S(x(0)) + \int_0^t J(x(\tau),u(\tau)) - \alpha(\|x(\tau)-\bar x\|_X + \|u(\tau)-\bar u\|_U) d\tau \]
		holds along all solutions. In finite dimensions this is shown using exponential reachability of $\bar x$, which is implied by stabilizability, in \cite[Theorem 2]{Faulwasser2017}. In infinite dimensions the implication strict dissipativity $\Rightarrow$ measure turnpike is analyzed in \cite[Theorem 2]{Trelat18}.
		
		\item Condition (ii) follows from exponential stabilizability of $\bar x$ by an affine linear feedback $u = Kx+\bar u$ and continuity of $J$. This can be seen straightforwardly by integrating $J$ along the exponentially stable closed loop solution.
		
		\item Condition (iii) follows from strict dissipativity and (exact) local controllability around $\bar x$. A proof in discrete time can be found in \cite[Sec.\ 6]{Grue13}. This proof easily carries over to the continuous time setting of this paper. 
		
		\item Condition (iv) follows from strict dissipativity if the $\alpha$ in the strict dissipativity formulation grows quickly enough and the storage function is continuous in $\bar x$. Continuity implies that $S(x(0))$ and $S(x(T))$ can be bounded by the $\sigma$-terms and the growth condition on $\alpha$ implies the inequality
		\[ \int_0^t \alpha(\|x(\tau)-\bar x\|_X + \|u(\tau)-\bar u\|_U) d\tau \ge \max_{t\in[0,T]} \sigma(\|x(t)-\bar x\|_X).\]
	\end{itemize}

All theorems in this paper provide the interval turnpike property for the adjoints if the adjoints are bounded for horizon $T$ tending to infinity. In other words, the theorems state that the adjoints are either unbounded or they satisfy the turnpike property.  It is thus necessary to establish a bound on the adjoints in order to conclude the turnpike property. We end this section by providing an example where such a bound can be deduced.
\begin{example}[Boundedness of adjoint]
	\label{ex:boundedness}
	We give an example with $X=H^1_0(\Omega)$ and $Y=L_2(\Omega)$, where the bound $\sup_{T\ge 0}\|\lambda-\bar{\lambda}\|_{C(t_2(T),T;H^1(\Omega))} < \infty$ assumed in \Cref{thm:expstab,thm:parstab} and \Cref{cor:conv_exact} holds. Consider the running cost $J(x,u) = \frac{1}{2}\|x-x_d\|^2_{L_2(\Omega_o)} + \frac{1}{2}\|u-u_d\|^2_{L_2(\Omega_c)}$, $\mathcal{A}=\Delta$ endowed with Dirichlet boundary conditions and that $f(x,u)=f(x)$ is monotonously non-increasing, i.e., $f'(x)\leq 0$ with $f(0)=0$. Throughout this example, we consider by $c\geq 0$ a $T$-independent generic constant. By optimality of $(x,u)$ we obtain for any $(x_r,u_r)$ satisfying the dynamics, that
	\begin{align*}
	\|x\|^2_{L_2((0,T)\times \Omega_o)}&\leq \int_0^T J(x(t),u(t))\,dt\leq  \int_0^TJ(x_r(t),u_r(t))\,dt\\&= \frac12\left(\|x_r-x_d\|^2_{L_2((0,T)\times \Omega_o)} + \|u_r-u_d\|^2_{L_2((0,T)\times \Omega_c}\right)
	\end{align*}
	Inserting this into the adjoint equation yields, using parabolic regularity and $f'(x)\leq 0$, that
	\begin{align}
	\label{ex:eq:adjointest}
	\|\lambda(t)\|_{H^1(\Omega)} &\leq c\|x\|_{L_2((0,T)\times \Omega_o)}\\\nonumber &\leq c\left(\|x_r-x_d\|_{L_2((0,T)\times \Omega_o)} + \|u_r-u_d\|_{L_2((0,T)\times \Omega_c}\right)
	\end{align}
	for all $t\in [0,T]$. Thus, we get
	\begin{align*}
	\|\lambda(t)-\bar{\lambda}&\|_{C(0,T;H^1(\Omega))}\\&\leq c\left(\|\bar{\lambda}\|_{H^1(\Omega)}+\|x_r-x_d\|_{L_2((0,T)\times \Omega_o)} + \|u_r-u_d\|_{L_2((0,T)\times \Omega_c}\right)
	\end{align*}
	which yields the result if
	\begin{align*}
	\int_0^TJ(x_r,u_r) = \frac12\left(\|x_r-x_d\|^2_{L_2((0,T)\times \Omega_o)} + \|u_r-u_d\|^2_{L_2((0,T)\times \Omega_c)}\right)
	\end{align*}
	can be bounded independently of $T$ for any $u_r\in L_2((0,T)\times \Omega_c)$ and corresponding state $x_r$.
	The same argumentation also carries over to the case of Neumann boundary control. The case of boundary observation can not be included, as we do not have the required regularity of the adjoint to deduce \eqref{ex:eq:adjointest}.
\end{example}
\end{remark}
	\section{Numerical example}
	\label{sec:example}
	We present an example with a semilinear heat equation with Neumann boundary control:
	\begin{align*}
	x'-\Delta x + x^3 &= 0\qquad &&\text{ in } [0,T]\times \Omega, \\
	\frac{\partial x}{\partial \nu}&=u\qquad &&\text{ in } [0,T]\times \partial \Omega,\\
	x(0)&=x_0 \qquad &&\text{ in }\Omega,
	\end{align*}
	where $\frac{\partial x}{\partial \nu}$ is the outward unit normal derivative.
	As a cost function, we consider
	\begin{align*}
	\int_0^T J(x(t),u(t))dt=\int_0^T \frac{1}{2}\|x(t)-x_d\|_{L_2(\Omega)}^2 + \frac{1}{2}\|u(t)\|_{L_2(\partial \Omega)}^2\,dt.
	\end{align*}
	We choose $\Omega = [0,3]\times [0,1]$ as the spatial domain and the horizon $T=10$.
	Additionally, we set $x_0=0$ and the reference trajectory defined by
	\begin{align*}
	x_d(\omega) &:= g\left(\frac{10}{3}\left\|\omega - \begin{pmatrix}
	1.5\\0.5\end{pmatrix}\right\|\right),\\
	\text{where}\quad\qquad g(s) &:= \begin{cases}
	10e^{1-\frac{1}{1-s^2}} \qquad &s< 1\\
	0   &\text{else}.
	\end{cases}
	\end{align*}
	This static reference is depicted in \Cref{fig:ref}. The optimal control problem is solved with the C++-library for vector space algorithms \textit{Spacy}\footnote{https://spacy-dev.github.io/Spacy/}  using the finite element library \textit{Kaskade7} \cite{Goetschel2020}.
	\begin{figure}[H]
		\centering
		\scalebox{0.8}{\input{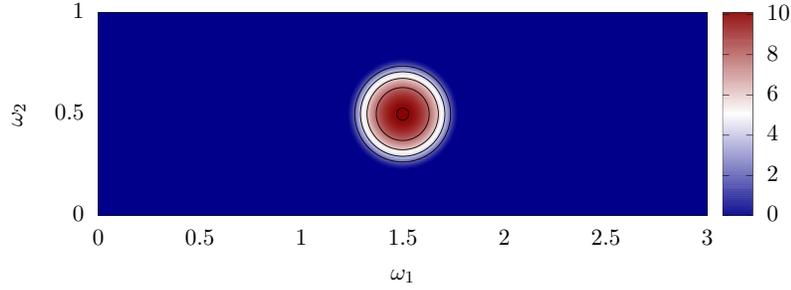}}
		\caption{Depiction of the static reference trajectory $x_d(\omega_1,\omega_2)$.}
		\label{fig:ref}
	\end{figure}
	For in-depth analysis of semilinear parabolic optimal control problems we refer the reader to \cite{Raymond1999} or \cite[Chap.\ 5]{Troeltzsch2010}. Considering the regularity of the static adjoint, for sufficiently smooth data we obtain that $\bar{\lambda}\in C(\bar{\Omega})$, cf.\ \cite{Casas1993}. We set $\mathcal{A}=\Delta$ and $\varphi(x)=x^3$ and denote the superposition operator corresponding to $\varphi'(x)=3x^2$ by $\Phi$.
	We numerically verify that the turnpike property of \Cref{as:tp} for the optimal state and control holds in $X=H^1(\Omega)$ and $U=L_2(\partial\Omega)$, cf. \Cref{fig:tp_state_control}. Thus, we depict in \Cref{fig:tp_state_control} the norm of state and adjoint over time and in \Cref{fig:turnpike_state_control2} a snapshot of the dynamic solution and compare it to the steady state solution.
	\begin{figure}[H]
		\centering
%
%
\definecolor{mycolor1}{rgb}{0.00000,0.44700,0.74100}%
\begin{tikzpicture}

\begin{axis}[%
width=4in,
height=0.793in,
at={(0in,1.2in)},
scale only axis,
xmin=-0.5,
xmax=10.5,
xtick={0,1,2,3,4,5,6,7,8,9,10},
xticklabels={\empty},
ymin=0,
ymax=1,
ylabel style={font=\color{white!15!black}},
ylabel={$\|x(t)\|_{H^1(\Omega)}$},
axis background/.style={fill=white},
title={Norm of optimal state over time}
]
\addplot [color=black, line width=2.0pt, forget plot]
  table[row sep=crcr]{%
0	0\\
0.1	0.248903645264741\\
0.2	0.292554245812156\\
0.3	0.327204602336484\\
0.4	0.364641477501348\\
0.5	0.401904131866604\\
0.6	0.436546363170142\\
0.7	0.467526396981906\\
0.8	0.49462345025298\\
0.9	0.518010171163363\\
1	0.538025246167201\\
1.1	0.555060075639231\\
1.2	0.569503954727109\\
1.3	0.581718934718247\\
1.4	0.592029810615113\\
1.5	0.600721785718475\\
1.6	0.60804192575653\\
1.7	0.614202346074852\\
1.8	0.619384041759536\\
1.9	0.623740791563295\\
2	0.627402851234184\\
2.1	0.630480308084337\\
2.2	0.633066053669631\\
2.3	0.635238376545315\\
2.4	0.637063199646975\\
2.5	0.63859599679353\\
2.6	0.639883425771434\\
2.7	0.640964714689105\\
2.8	0.641872835656846\\
2.9	0.642635496414479\\
3	0.643275976890204\\
3.1	0.643813834148302\\
3.2	0.644265495928094\\
3.3	0.64464476005826\\
3.4	0.644963214463257\\
3.5	0.645230590248757\\
3.6	0.645455058433629\\
3.7	0.645643479254062\\
3.8	0.645801611567484\\
3.9	0.645934288697439\\
4	0.646045566056115\\
4.1	0.64613884503223\\
4.2	0.646216976915326\\
4.3	0.646282350022944\\
4.4	0.64633696268736\\
4.5	0.646382484328716\\
4.6	0.646420306478687\\
4.7	0.646451585312645\\
4.8	0.646477276989324\\
4.9	0.646498166877582\\
5	0.646514893563202\\
5.1	0.646527968369338\\
5.2	0.646537790987174\\
5.3	0.646544661694546\\
5.4	0.646548790536041\\
5.5	0.646550303745294\\
5.6	0.64654924760593\\
5.7	0.64654558986942\\
5.8	0.646539218773511\\
5.9	0.646529939631726\\
5.99999999999999	0.646517468890286\\
6.09999999999999	0.646501425471741\\
6.19999999999999	0.646481319141932\\
6.29999999999999	0.646456535546381\\
6.39999999999999	0.646426317461088\\
6.49999999999999	0.646389741687732\\
6.59999999999999	0.646345690891364\\
6.69999999999999	0.64629281952545\\
6.79999999999999	0.646229512810302\\
6.89999999999999	0.646153837520996\\
6.99999999999999	0.646063483093366\\
7.09999999999999	0.64595569126412\\
7.19999999999999	0.645827172114252\\
7.29999999999999	0.645674003972542\\
7.39999999999999	0.645491514144112\\
7.49999999999999	0.645274136840183\\
7.59999999999999	0.645015243976286\\
7.69999999999999	0.644706943646635\\
7.79999999999999	0.644339840029534\\
7.89999999999999	0.643902747171408\\
7.99999999999999	0.643382347444331\\
8.09999999999999	0.642762783334352\\
8.19999999999999	0.642025168376137\\
8.29999999999999	0.641146999152422\\
8.39999999999999	0.640101444756375\\
8.49999999999999	0.638856482050519\\
8.59999999999999	0.637373832966191\\
8.69999999999999	0.635607641638938\\
8.79999999999998	0.633502800841153\\
8.89999999999998	0.630992794044465\\
8.99999999999998	0.627996855880582\\
9.09999999999998	0.624416166900995\\
9.19999999999998	0.620128700116674\\
9.29999999999998	0.614982287887041\\
9.39999999999998	0.608785683900332\\
9.49999999999998	0.601298461590731\\
9.59999999999998	0.592224215346476\\
9.69999999999998	0.581222503554813\\
9.79999999999998	0.567987617163065\\
9.89999999999998	0.55254300178868\\
9.99999999999998	0.536208058426657\\
};
\end{axis}

\begin{axis}[%
width=4in,
height=0.793in,
at={(0in,0in)},
scale only axis,
xmin=-0.5,
xmax=10.5,
xtick={0,1,2,3,4,5,6,7,8,9,10},
ylabel style={font=\color{white!15!black}},
xlabel = {time $t$},
ylabel={$\|u(t)\|_{L_2(\partial\Omega)}$},
axis background/.style={fill=white},
title={Norm of optimal control over time}
]
\addplot [color=black, line width=2.0pt, forget plot]
  table[row sep=crcr]{%
0	0\\
0.1	0.652900529530023\\
0.2	0.56963403315228\\
0.3	0.501758840997319\\
0.4	0.44696949905044\\
0.5	0.403165363665998\\
0.6	0.368441192160381\\
0.7	0.341112573348394\\
0.8	0.319726384250835\\
0.9	0.303056966782957\\
1	0.290092259160831\\
1.1	0.280012917440207\\
1.2	0.272167150679186\\
1.3	0.266044002997808\\
1.4	0.261247473204619\\
1.5	0.257473147561243\\
1.6	0.254488211718449\\
1.7	0.252115041083533\\
1.8	0.250218136021358\\
1.9	0.248693950948529\\
2	0.247463101699509\\
2.1	0.246464459670266\\
2.2	0.245650706674075\\
2.3	0.244985001777402\\
2.4	0.244438485108993\\
2.5	0.243988407110902\\
2.6	0.243616723242747\\
2.7	0.243309034490701\\
2.8	0.243053784849951\\
2.9	0.242841650112261\\
3	0.242665069511349\\
3.1	0.242517884491296\\
3.2	0.242395058206665\\
3.3	0.242292456214711\\
3.4	0.24220667384143\\
3.5	0.242134899386057\\
3.6	0.242074805034713\\
3.7	0.24202445934827\\
3.8	0.241982256664684\\
3.9	0.241946859852342\\
4	0.241917153669846\\
4.1	0.241892206603024\\
4.2	0.241871239515132\\
4.3	0.241853599800402\\
4.4	0.24183874000244\\
4.5	0.241826200068411\\
4.6	0.241815592572662\\
4.7	0.241806590370605\\
4.8	0.241798916243739\\
4.9	0.241792334175747\\
5	0.24178664196229\\
5.1	0.241781664907034\\
5.2	0.241777250396054\\
5.3	0.241773263174224\\
5.4	0.2417695811719\\
5.5	0.241766091749424\\
5.6	0.24176268824145\\
5.7	0.241759266693557\\
5.8	0.241755722690498\\
5.9	0.241751948178991\\
5.99999999999999	0.241747828188403\\
6.09999999999999	0.241743237349985\\
6.19999999999999	0.241738036109315\\
6.29999999999999	0.241732066516925\\
6.39999999999999	0.241725147468199\\
6.49999999999999	0.241717069244348\\
6.59999999999999	0.241707587180155\\
6.69999999999999	0.241696414248746\\
6.79999999999999	0.24168321230494\\
6.89999999999999	0.241667581661074\\
6.99999999999999	0.24164904857299\\
7.09999999999999	0.24162705007493\\
7.19999999999999	0.241600915397312\\
7.29999999999999	0.24156984289514\\
7.39999999999999	0.241532870950566\\
7.49999999999999	0.241488840601867\\
7.59999999999999	0.241436346552609\\
7.69999999999999	0.241373671507388\\
7.79999999999999	0.241298696114769\\
7.89999999999999	0.241208772621533\\
7.99999999999999	0.241100543782808\\
8.09999999999999	0.240969678251727\\
8.19999999999999	0.240810477412841\\
8.29999999999999	0.24061528298927\\
8.39999999999999	0.240373574329237\\
8.49999999999999	0.240070580554495\\
8.59999999999999	0.239685132424692\\
8.69999999999999	0.239186321183591\\
8.79999999999998	0.238528285104626\\
8.89999999999998	0.23764206111019\\
8.99999999999998	0.23642284837379\\
9.09999999999998	0.234710134087408\\
9.19999999999998	0.232256799142516\\
9.29999999999998	0.228681411800296\\
9.39999999999998	0.223395361127026\\
9.49999999999998	0.215493586726853\\
9.59999999999998	0.203596221652948\\
9.69999999999998	0.185636255935455\\
9.79999999999998	0.158640228654919\\
9.89999999999998	0.118817749080512\\
9.99999999999998	0.0635849818314392\\
};
\end{axis}

\end{tikzpicture}%
		\caption{Turnpike property for optimal state and control.}
		\label{fig:tp_state_control}
	\end{figure}
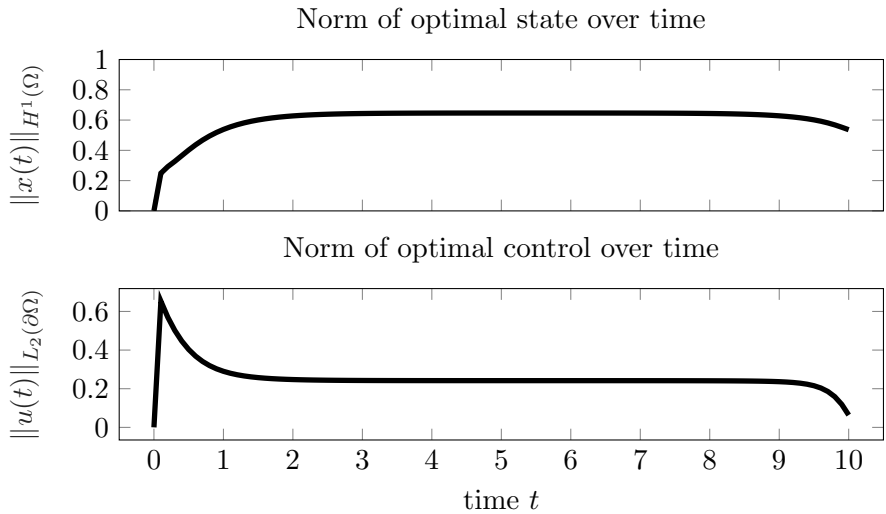
	\begin{figure}[H]
		\begin{minipage}[]{0.4\linewidth}
			\centering
			Dynamic control $u$ at $t=5$ (above) vs steady control $\bar{u}$ (below)
			\includegraphics[width=\linewidth]{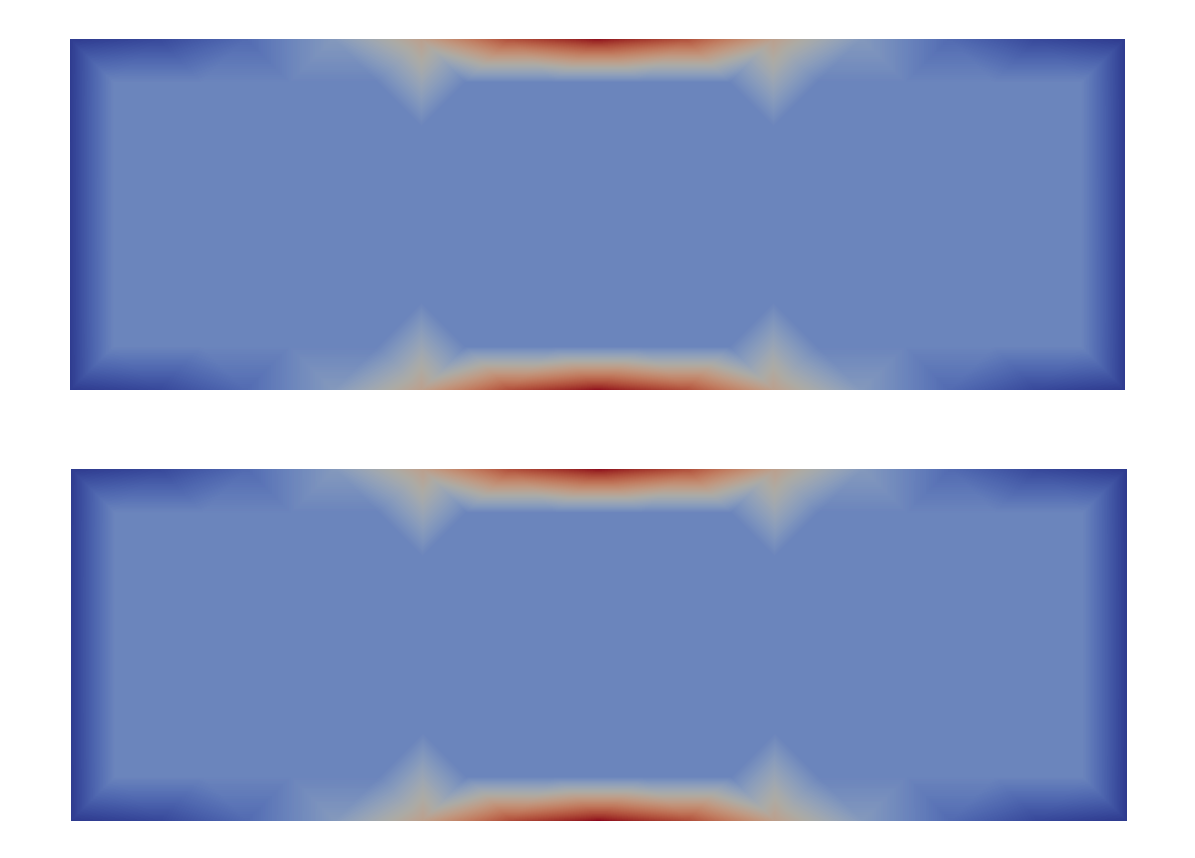}
		\end{minipage}
		\begin{minipage}[]{0.4\linewidth}
			\centering
			Dynamic state $x$ at $t=5$ (above) vs steady state $\bar{x}$ (below)
			\includegraphics[width=\linewidth]{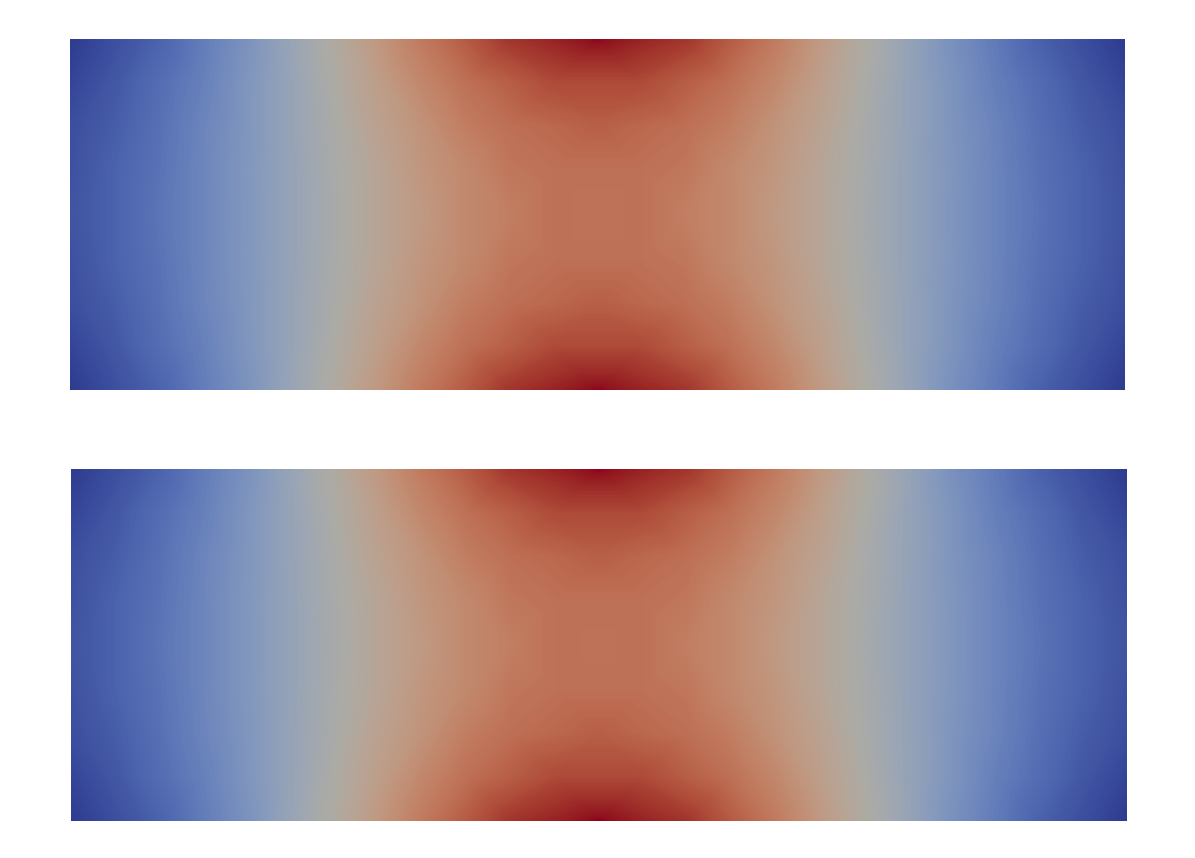}
		\end{minipage}
		\caption{Comparison of dynamic solution $(x,u)$ in the middle of the time interval with the steady state solution $(\bar{x},\bar{u})$ with the same coloring as used in \Cref{fig:tp_state_control}.}
		\label{fig:turnpike_state_control2}
	\end{figure}
	While the turnpike property of the state and adjoints is verified numerically, the remaining assumptions can be checked analytically as follows.
	By the classical embeddings $H^1(\Omega)\hookrightarrow L_p(\Omega)$ for any $1\leq p < \infty$ for $\Omega \subset \mathbb{R}^2$, cf.\ \cite[Sec.\ 7.1]{Troeltzsch2010} or \cite[Chap.\ V]{Adams1975} and as the nonlinearity is cubic, the occurring superposition operators satisfy \Cref{as:cont} for $Y=L_2(\Omega)$. Moreover, as $D(\mathcal{A}^*)=\{v\in H^2(\Omega)\,|\, \frac{\partial}{\partial \nu}v=0\}\hookrightarrow H^1(\Omega)=X$ compactly, the perturbation $\Phi(\bar{x})\in L(X,Y)$ is $\mathcal{A}^*$-compact and $A^*=\mathcal{A}^* +\Phi(\bar{x})^*$ generates an analytic semigroup on $L_2(\Omega)$, and $D(A^*)=D(\mathcal{A}^*)$, cf.\ \cite[Chap.\ III, Thm.\ 2.10]{Engel2000}. Thus, assuming additionally a boundedness condition of the adjoint, \Cref{thm:parstab} applies and we obtain the turnpike property for the adjoint in $Y=L_2(\Omega)$. Additionally by analyticity of the semigroup, the improved estimate of \Cref{thm:analytic} holds and we get the estimate also in the space $W^{A^*}(0,T)$ with $D(A^*)=\{v\in H^2(\Omega)\,|\, \frac{\partial}{\partial \nu}v=0\}$ and $(D(A^*),L_2(\Omega))_{\frac{1}{2}}\subset H^1(\Omega)$.
	In \Cref{fig:tp_adjoint} we observe the turnpike property for the adjoint in the $H^1(\Omega)$-norm and also in the $L_2(\Omega)$-norm.
	\begin{figure}[H]
		\centering
%
%
\definecolor{mycolor1}{rgb}{0.00000,0.44700,0.74100}%
\begin{tikzpicture}

\begin{axis}[%
width=4in,
height=0.793in,
at={(0in,1.2in)},
scale only axis,
xmin=-0.5,
xmax=10.5,
xtick={0,1,2,3,4,5,6,7,8,9,10},
xticklabels={\empty},
ylabel style={font=\color{white!15!black}},
ylabel={$\|\lambda(t)\|_{L_2(\Omega)}$},
axis background/.style={fill=white},
title={$L_2(\Omega)$-norm of adjoint over time}
]
\addplot [color=black, line width=2.0pt, forget plot]
  table[row sep=crcr]{%
0	0.482061869157957\\
0.1	0.482061869157956\\
0.2	0.430003913619551\\
0.3	0.387076954069737\\
0.4	0.352020391594401\\
0.5	0.323588640741803\\
0.6	0.300637858827959\\
0.7	0.282164622379158\\
0.8	0.267315254575738\\
0.9	0.255379314065899\\
1	0.24577527273099\\
1.1	0.238033051329498\\
1.2	0.231776150409559\\
1.3	0.226704935085748\\
1.4	0.22258186401405\\
1.5	0.21921894704107\\
1.6	0.216467400804886\\
1.7	0.214209297319241\\
1.8	0.212350924518459\\
1.9	0.210817563802286\\
2	0.209549410265436\\
2.1	0.208498397692837\\
2.2	0.207625730995909\\
2.3	0.206899967289856\\
2.4	0.206295520427532\\
2.5	0.205791491684239\\
2.6	0.205370751655987\\
2.7	0.2050192159848\\
2.8	0.204725271088948\\
2.9	0.204479316453889\\
3	0.204273397929452\\
3.1	0.204100912455871\\
3.2	0.203956369163976\\
3.3	0.203835195218824\\
3.4	0.203733577373821\\
3.5	0.203648332179659\\
3.6	0.203576799303856\\
3.7	0.203516753577909\\
3.8	0.20346633228623\\
3.9	0.20342397490824\\
4	0.203388373070316\\
4.1	0.203358428893271\\
4.2	0.203333220260729\\
4.3	0.203311971804145\\
4.4	0.203294030616687\\
4.5	0.203278845882287\\
4.6	0.203265951746743\\
4.7	0.203254952871655\\
4.8	0.203245512204445\\
4.9	0.20323734057286\\
5	0.203230187773351\\
5.1	0.20322383487209\\
5.2	0.203218087477095\\
5.3	0.203212769771439\\
5.4	0.20320771912214\\
5.5	0.203202781097847\\
5.6	0.203197804741579\\
5.7	0.203192637953083\\
5.8	0.203187122838977\\
5.9	0.203181090888089\\
5.99999999999999	0.203174357824046\\
6.09999999999999	0.203166717977175\\
6.19999999999999	0.203157938002652\\
6.29999999999999	0.203147749751008\\
6.39999999999999	0.203135842069657\\
6.49999999999999	0.203121851278787\\
6.59999999999999	0.203105350020005\\
6.69999999999999	0.203085834118985\\
6.79999999999999	0.203062707030656\\
6.89999999999999	0.203035261342024\\
6.99999999999999	0.20300265668639\\
7.09999999999999	0.202963893262742\\
7.19999999999999	0.202917779939766\\
7.29999999999999	0.202862895631447\\
7.39999999999999	0.20279754222497\\
7.49999999999999	0.202719686767228\\
7.59999999999999	0.20262688979076\\
7.69999999999999	0.202516215455958\\
7.79999999999999	0.202384117407976\\
7.89999999999999	0.202226291590957\\
7.99999999999999	0.202037483260664\\
8.09999999999999	0.201811229362924\\
8.19999999999999	0.201539508172462\\
8.29999999999999	0.201212253864235\\
8.39999999999999	0.200816671798905\\
8.49999999999999	0.200336256527088\\
8.59999999999999	0.19974936232106\\
8.69999999999999	0.19902709534714\\
8.79999999999998	0.198130171930856\\
8.89999999999998	0.197004195205687\\
8.99999999999998	0.195572507365882\\
9.09999999999998	0.193725324226859\\
9.19999999999998	0.191303175967103\\
9.29999999999998	0.188071650846934\\
9.39999999999998	0.18368289881566\\
9.49999999999998	0.177617007270466\\
9.59999999999998	0.169092517195575\\
9.69999999999998	0.15692754230257\\
9.79999999999998	0.139309357848771\\
9.89999999999998	0.113319529508732\\
9.99999999999998	0.0733092627947845\\
};
\end{axis}

\begin{axis}[%
width=4in,
height=0.793in,
at={(0in,0in)},
scale only axis,
xmin=-0.5,
xmax=10.5,
xtick={0,1,2,3,4,5,6,7,8,9,10},
ylabel style={font=\color{white!15!black}},
ylabel={$\|\lambda(t)\|_{H^1(\Omega)}$},
xlabel={time $t$},
axis background/.style={fill=white},
title={$H^1(\Omega)$-norm of adjoint over time}
]
\addplot [color=black, line width=2.0pt, forget plot]
  table[row sep=crcr]{%
0	0.718936220300763\\
0.1	0.718936220300759\\
0.2	0.681001290547671\\
0.3	0.650573804887459\\
0.4	0.62641815838298\\
0.5	0.607335328013053\\
0.6	0.59228202238474\\
0.7	0.580396665740955\\
0.8	0.570988722573438\\
0.9	0.563514782620392\\
1	0.557551860611943\\
1.1	0.552772597770235\\
1.2	0.548924152371452\\
1.3	0.545811100710761\\
1.4	0.543281999068559\\
1.5	0.54121901758937\\
1.6	0.539530031770056\\
1.7	0.538142624368975\\
1.8	0.536999546475465\\
1.9	0.536055281466381\\
2	0.535273437862482\\
2.1	0.534624763745037\\
2.2	0.534085627335982\\
2.3	0.533636847893025\\
2.4	0.533262790750714\\
2.5	0.532950662417814\\
2.6	0.532689957985867\\
2.7	0.53247202517581\\
2.8	0.532289718263312\\
2.9	0.532137121713531\\
3	0.532009328239785\\
3.1	0.531902259632437\\
3.2	0.531812521416467\\
3.3	0.53173728443167\\
3.4	0.531674187965472\\
3.5	0.531621260234571\\
3.6	0.531576852902603\\
3.7	0.531539587006425\\
3.8	0.53150830819433\\
3.9	0.531482049593346\\
4	0.531460000947634\\
4.1	0.531441482926702\\
4.2	0.531425925706223\\
4.3	0.531412851087568\\
4.4	0.531401857553414\\
4.5	0.531392607762957\\
4.6	0.531384818076423\\
4.7	0.531378249768823\\
4.8	0.531372701650329\\
4.9	0.531368003857759\\
5	0.531364012620367\\
5.1	0.531360605834917\\
5.2	0.531357679311198\\
5.3	0.531355143570602\\
5.4	0.531352921097893\\
5.5	0.531350943960528\\
5.6	0.531349151721244\\
5.7	0.531347489578457\\
5.8	0.5313459066756\\
5.9	0.531344354524904\\
5.99999999999999	0.531342785493348\\
6.09999999999999	0.531341151298313\\
6.19999999999999	0.531339401457554\\
6.29999999999999	0.531337481631659\\
6.39999999999999	0.531335331786102\\
6.49999999999999	0.531332884082579\\
6.59999999999999	0.531330060382546\\
6.69999999999999	0.531326769205874\\
6.79999999999999	0.531322901927353\\
6.89999999999999	0.531318327903968\\
6.99999999999999	0.531312888090912\\
7.09999999999999	0.531306386501479\\
7.19999999999999	0.531298578559932\\
7.29999999999999	0.53128915493364\\
7.39999999999999	0.531277718729322\\
7.49999999999999	0.531263752873209\\
7.59999999999999	0.531246572875705\\
7.69999999999999	0.53122525771677\\
7.79999999999999	0.531198547835186\\
7.89999999999999	0.53116469348671\\
7.99999999999999	0.531121228023179\\
8.09999999999999	0.531064627370171\\
8.19999999999999	0.530989796766498\\
8.29999999999999	0.530889295072394\\
8.39999999999999	0.53075216021939\\
8.49999999999999	0.530562128504664\\
8.59999999999999	0.530294933270179\\
8.69999999999999	0.529914207133377\\
8.79999999999998	0.529365270333824\\
8.89999999999998	0.528565728941235\\
8.99999999999998	0.527391279521269\\
9.09999999999998	0.525654353065473\\
9.19999999999998	0.523072141038719\\
9.29999999999998	0.519219003413304\\
9.39999999999998	0.513456004836167\\
9.49999999999998	0.504826460096589\\
9.59999999999998	0.491896611515997\\
9.69999999999998	0.47248200960828\\
9.79999999999998	0.442998945721979\\
9.89999999999998	0.395935092994185\\
9.99999999999998	0.304800339357979\\
};
\end{axis}

\end{tikzpicture}%
		\caption{Turnpike property for the adjoint.}
		\label{fig:tp_adjoint}
	\end{figure}
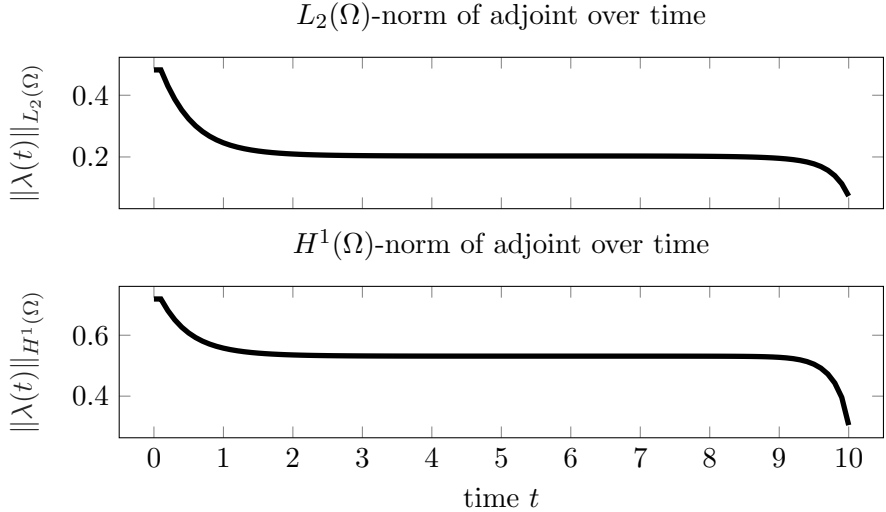
	\section{Summary and outlook}
	This paper presented an interval turnpike result for the adjoints of finite- and infinite-dimensional nonlinear optimal control problems
	under the assumption of an interval turnpike on states and controls.
	We analyzed the case of stabilizable dynamics 
	governed by a generator of a semigroup with finite dimensional unstable part satisfying a spectral decomposition
	condition. We have shown the desired turnpike property under continuity
	assumptions on the first-order optimality conditions. We illustrated our findings drawing upon a boundary controlled semilinear heat equation.
	
	We briefly discuss possible directions of further research. A central assumption in the results of this work is the boundedness condition on the adjoint, i.e., $\sup_{T\ge 0}\|\lambda-\bar{\lambda}\|_{C(t_2(T),T;X)} < \infty$. We have presented a very particular case in \Cref{ex:boundedness} where this bound holds, however, under the strict condition that the cost at the equilibrium has to vanish. It is desirable to prove the bound under milder conditions.
	
	 In \Cref{sec:example} we verified the turnpike property of state and control only numerically. This is due to the fact that the verification of assumptions of global turnpike theorems via dissipativity, cf.\ \Cref{rem:dissipativity}, is a highly non-trivial issue. In that context, suitable storage functions need to be constructed and we refer to \cite{Gruene2019a} for a promising approach in that direction.

 An inspection of the proofs of \Cref{sec:stableorcontrollable,sec:finitedim} shows that, under suitable assumptions, it should be possible replace the control operator $\mathcal{B}$ by an unbounded operator, which is, e.g., the case for boundary control. In that context, a concept like admissibility or assuming that solutions are classical solutions is a suitable replacement for boundedness of $\mathcal{B}$.

	\bibliographystyle{abbrv}
	\bibliography{literature}
	
\end{document}